\documentclass[11pt]{article}
\usepackage[a4paper,margin=25mm]{geometry}
\usepackage[T1]{fontenc}
\usepackage[utf8]{inputenc}
\usepackage{lmodern}
\usepackage{amsmath,amssymb,amsthm,mathtools}
\usepackage{booktabs,longtable}
\usepackage{flafter}
\usepackage{tikz,pgfplots}
\pgfplotsset{compat=1.18}
\usepackage{microtype}
\usepackage{indentfirst}
\usepackage{needspace}
\usepackage{xurl}
\usepackage[colorlinks=true,linkcolor=black,citecolor=black,urlcolor=black]{hyperref}
\hypersetup{
  pdftitle={Orthodox queen domination: finite constructions and an asymptotic density gap},
  pdfauthor={Yixiang Kong},
  pdfsubject={Queen domination and a density gap for orthodox covers},
  pdfkeywords={queen domination, orthodox cover, p-cover, asymptotic density, rational certificate}
}
\numberwithin{equation}{section}
\newtheorem{theorem}{Theorem}[section]
\newtheorem{lemma}[theorem]{Lemma}
\newtheorem{proposition}[theorem]{Proposition}
\newtheorem{corollary}[theorem]{Corollary}
\newcommand{\doi}[1]{\href{https://doi.org/#1}{\nolinkurl{doi:#1}}}
\allowdisplaybreaks[1]
\title{Orthodox queen domination:\\
finite constructions and an asymptotic density gap}
\author{Yixiang Kong\\[0.4ex]
\small Sydney Smart Technology College, Northeastern University\\
\small Qinhuangdao, China\\
\small \texttt{202319280@stu.neuq.edu.cn}}
\date{September 2026}

\begin{document}
\maketitle
\begin{abstract}
An orthodox dominating set on an \(n\times n\) chessboard occupies every row of
one parity and every column of a possibly different parity. For
\(n\ge400001\), every such set contains more than
\((1/2+1/80000)n-2\) queens, on odd and even boards and with attacking or
boundary queens allowed.
Second-moment estimates and an exact rational line-weight certificate give a
stronger bound for \(p\)-covers; finite diagonal completion and board extension
transfer it to orthodox covers. The cost of extending an arbitrary dominating
set to an orthodox cover yields an inequality with an explicit defect term.
The previously constructed independent, border-free Type-A \(1\)-cover of
\(Q_{221}\) with \(111\) queens supplies a finite seed. Classical
amplification gives ordinary and independent domination upper bounds with
coefficients \(112/221\) and \(113/221\), respectively. Its order 221 is below
the density threshold 400001. For admissible seeds whose orders tend to
infinity, the lower limit of these coefficients is at least \(1/2+1/16000\).
\end{abstract}

\noindent\textbf{Keywords:} queen domination; orthodox cover; \(p\)-cover;
asymptotic density; exact certificate.

\smallskip
\noindent\textbf{Mathematics Subject Classification (2020):} 05C69, 68R05.

\section{Introduction}\label{sec:1}

The queen's graph \(Q_n\) has one vertex for each square of the \(n\times n\)
chessboard, with adjacency determined by a common row, column, or diagonal.
A dominating set covers every square, and an independent dominating set also
consists of pairwise nonattacking queens. Their minimum sizes are denoted by
\(\gamma(Q_n)\) and \(i(Q_n)\).

Structured configurations called \(p\)-covers supply exact finite examples
and seeds for asymptotic upper-bound constructions~\cite{ow2001,weakley2002upper}.
Orthodox covers form a larger family: they occupy every row of one parity and
every column of a possibly different parity, and dominate the board
(\cite[Definition~3.2]{neuhaus2009}; \cite[Section~3.1]{bozoki2019}).
The lower-bound problem considered here concerns the excess above half the
board order under these parity constraints. Throughout the paper, density
means the number of queens divided by the side length \(n\).

The Parallelogram Law relates squared line indices
(\cite[Theorem~4]{ow2001}; \cite[Theorem~1]{weakley2002lower}).
Finozhenok and Weakley~\cite{finozhenok2007} combine an absolute-value identity
at each queen with diagonal capacities and exact saturation to exclude sets of
size \((n-1)/2\) outside two exceptional boards. Weakley~\cite{weakley2022}
analyzes even boards at exact size \(n/2\), and Neuhaus
\cite[Section~3.1 and Corollary~3.16]{neuhaus2009} develops related structure
at size \((n+1)/2\). In the present estimates, additional queens and repeated
line indices contribute explicit error terms. These estimates yield an
exclusion interval whose width is proportional to the board order.
Karandikar and Dutta~\cite{karandikar2024} studied a line-covering relaxation.
The certificate below retains the common queen position when the four line
weights are evaluated.

\paragraph{Density results.}
Let \(\gamma_{\mathrm{orth}}(n)\) be the minimum size of an orthodox cover of
\(Q_n\). The principal result, Theorem~\ref{thm:2.3}, is the orthodox
density bound
\begin{equation}\label{eq:1.1}
\gamma_{\mathrm{orth}}(n)>\frac{40001}{80000}n-2
\quad(n\ge400001),\qquad
\liminf_{n\to\infty}\frac{\gamma_{\mathrm{orth}}(n)}n
\ge\frac12+\frac1{80000}.
\end{equation}
Theorem~\ref{thm:2.1} gives the stronger density estimate for the smaller
family of \(p\)-covers, with minimum size \(m_p(n)\):
\begin{equation}\label{eq:1.2}
\liminf_{\substack{n\to\infty\\n\text{ odd}}}\frac{m_p(n)}n
\ge\frac{8001}{16000}=\frac12+\frac1{16000},\qquad p\in\{0,1\}.
\end{equation}
These two lower bounds are the central density results of the paper. Their
domains are the orthodox and parity-cover classes specified in Section~\ref{sec:2}.

\paragraph{Proof tools.}
The proof of~\eqref{eq:1.2} combines a second-moment estimate with a weighted
count. The moment estimate bounds the required diagonal ranges of the whole
configuration. A large subset with pairwise distinct row, column, and diagonal
indices is then used in the weighted count. The resulting line capacities
force a positive excess above half the board order.
To obtain~\eqref{eq:1.1}, the missing diagonal coverage of a relaxed
\(p\)-cover is completed by adding queens. The relevant intersections form
two bipartite graphs with nested neighborhoods. Their degrees, and hence
the completion cost, are controlled by the additional occupied rows and
columns. A board extension by at most two rows and columns aligns the two
parities. These are the intermediate counting and reduction steps through
which the parity-cover estimate reaches the full orthodox class.

\paragraph{Finite example and consequences.}
The independent, border-free Type-A \(1\)-cover of \(Q_{221}\) with \(111\)
queens from the author's earlier preprint~\cite{kong2026} supplies a concrete
instance of the Type-A family and its amplification formula.
Neuhaus~\cite[Appendix~B]{neuhaus2009} already established
\(\gamma(Q_{221})=111\). Independence and the border-free Type-A conditions
are verified by the coordinate certificate. Thus \(i(Q_{221})=111\), and
classical amplification (\cite[Theorem~5]{ow2001}; \cite{weakley2002upper})
gives the previously reported bounds~\cite{kong2026}
\begin{equation}\label{eq:1.3}
\gamma(Q_N)\le\frac{112}{221}N+O(1),\qquad
i(Q_N)\le\frac{113}{221}N+O(1).
\end{equation}
These are upper bounds for ordinary and independent domination on the
amplified board. The seed order 221 is far below the threshold 400001 in the
density theorems; its role here is a fixed finite input to the classical
construction.
Each application holds the seed fixed while the amplified board order \(N\)
grows. A second limit concerns the choice of seed: for an admissible seed of
order \(n\) and size \(d\), the coefficients are \((d+1)/n\) and \((d+2)/n\).
Theorem~\ref{thm:2.1} gives a lower limit of at least \(8001/16000\) as the
seed order tends to infinity. This connects the finite construction with the
density gap for the class from which the seeds are drawn. The restriction
on growing seeds is a consequence of the density theorem.

The minimum cost of an orthodox extension is also determined, giving the
defect inequality of Theorem~\ref{thm:2.4} for arbitrary dominating sets.
Further logical consequences concern the zero-cover construction route and
the incompatibility of two conjectures of Neuhaus. These applications are
developed after the density proofs, in Sections~\ref{sec:9}--\ref{sec:11}.

\paragraph{Verification and organization.}
Exact Python checks and a Lean formalization support the stated results;
their scope and accompanying files are described together in
Section~\ref{sec:12}. The finite certificate argument is given in
the proof of Lemma~\ref{lem:5.1}.

Section~\ref{sec:2} fixes the definitions and states the results in proof
order. Sections~\ref{sec:3}--\ref{sec:7} establish the \(p\)-cover gap
through representative selection, moment bounds, and weighted capacities.
Sections~\ref{sec:8}--\ref{sec:9} complete the reduction to orthodox covers
and derive the defect inequality. Section~\ref{sec:10} treats the finite
seed and the two amplification limits; Section~\ref{sec:11} develops the
logical consequences and remaining questions.
Appendices~\ref{app:knots} and~\ref{app:coordinates} contain the full
weight and coordinate certificates used by the proofs and verification.

\section{Definitions and main theorems}\label{sec:2}

The definitions below specify the classes to which the density bounds
apply. Theorems~\ref{thm:2.1}--\ref{thm:2.3} are stated in proof order:
the parity-cover estimate passes through relaxed covers to the principal
orthodox bound. Theorem~\ref{thm:2.4} then records its defect-form consequence
for arbitrary dominating sets.

On a board indexed by \(0,\ldots,n-1\), a dominating set is an
\textbf{orthodox cover} if every column of some parity \(p\) and every row of
some parity \(q\) are occupied, where \(p,q\in\{0,1\}\). When \(p=q\), it is
a \textbf{relaxed \(p\)-cover}~\cite[Definitions~3.2--3.3]{neuhaus2009}.
These definitions allow arbitrary cardinality and attacking queens.
Rows and columns will also be called orthogonals.

For an odd board, centered integer coordinates are used. Write
\[
k=\frac{n-1}{4},\qquad X,Y\in\{-2k,-2k+1,\ldots,2k\}.
\]

The scale \(k\) is an integer or a half-integer. A set is
\(p\)-orthodox if every row and column of parity \(p\) is occupied. A
\(p\)-orthodox set is a \textbf{\(p\)-cover} if every square whose two coordinates have parity
\(1-p\) also lies on an occupied diagonal. The occupied orthogonals cover
the remaining squares, so every \(p\)-cover is a dominating set.
Recentering shifts the row and column parities together; the same definitions
therefore apply in centered coordinates, with the parity labels adjusted.

The two long diagonals have equations \(Y-X=0\) and \(X+Y=0\).
For a set meeting both long diagonals, define \(e,f,u\) as
in~\cite{ow2001}. The maximal central intervals of occupied even difference
and sum indices are \(2i\), \(|i|\le e\), and \(2i\), \(|i|\le f\),
respectively. Beyond these intervals, \(u\) is the largest nonnegative
integer such that the difference diagonals \(\pm(2e+4j)\) and sum diagonals
\(\pm(2f+4j)\) are all occupied for \(1\le j\le u\).

The characterization in \cite[Theorem~3]{ow2001} states that a \(p\)-orthodox set meeting both long diagonals is a \(p\)-cover precisely when it satisfies one of
\begin{equation}\label{eq:2.1}
\begin{array}{ll}
\text{Type A:}&e+f\equiv p\pmod2,\quad e+f+2u\ge2k-2,\\
\text{Type B:}&e+f\equiv1-p\pmod2,\quad e+f\ge2k-1.
\end{array}
\end{equation}

In the Type-A argument, the specified diagonal families and the parity rows
and columns are called required lines.

\begin{theorem}\label{thm:2.1}
Let \(n\ge400001\) be odd and \(p\in\{0,1\}\). Every \(p\)-cover \(C\) of \(Q_n\) satisfies
\begin{equation}\label{eq:2.2}
|C|>\frac{8001}{16000}(n-1)-1.
\end{equation}

If \(C\) meets both long diagonals, the final subtraction of one can be omitted. Attacking queens and queens on the board boundary are allowed.
\end{theorem}

The asymptotic bound~\eqref{eq:1.2} follows by dividing~\eqref{eq:2.2} by
\(n\) and taking a lower limit. The proof first treats sets meeting both
long diagonals. Adding a central queen then gives the general case.

The completion argument uses the constants
\[
c_0=\frac{8001}{16000},\qquad c_1=\frac{c_0+2}{5}=\frac{40001}{80000}.
\]

\begin{theorem}\label{thm:2.2}
Let \(n\ge400001\) be odd. Every relaxed \(p\)-cover \(C\) of \(Q_n\) satisfies
\begin{equation}\label{eq:2.3}
|C|>c_1(n-1)-\frac15.
\end{equation}
\end{theorem}

\begin{theorem}\label{thm:2.3}
For every integer \(n\ge400001\), every orthodox cover \(C\) of \(Q_n\) satisfies
\begin{equation}\label{eq:2.4}
|C|>c_1n-2.
\end{equation}
\end{theorem}

Theorem~\ref{thm:2.3} includes all four choices of row and column parities.
Theorems~\ref{thm:2.2} and~\ref{thm:2.3} follow from
Theorem~\ref{thm:2.1} by finite completion and parity alignment in
Sections~\ref{sec:8} and~\ref{sec:9}.

For an arbitrary dominating set \(C\), let \(a_p(C)\) count its unoccupied columns of parity \(p\), and let \(b_q(C)\) count its unoccupied rows of parity \(q\). Define
\[
\delta(C)=\min_{p,q\in\{0,1\}}\max\{a_p(C),b_q(C)\}.
\]

\begin{theorem}\label{thm:2.4}
For every \(n\ge400001\) and every dominating set \(C\) of \(Q_n\),
\begin{equation}\label{eq:2.5}
|C|+\delta(C)>c_1n-2.
\end{equation}
\end{theorem}

The quantity \(\delta(C)\) is exactly the minimum number of additional queens required to extend \(C\) to an orthodox cover. The proof appears in Section~\ref{sec:9}.

\section{Extracting distinct representatives}\label{sec:3}

The weighted count will be applied to a subset with distinct indices within
each of the four line families. This section supplies that intermediate
counting tool and its Type-A size bound~\eqref{eq:3.1}, used in
Section~\ref{sec:6}. The subset is extracted by intersecting four sets of
representatives; the following elementary bound controls the loss.

\begin{lemma}\label{lem:3.1}
Suppose \(C\) consists of \(N\) points. For each of four families of parallel lines, let \(R_j\) distinct required lines meet \(C\). There is a subset \(G\subseteq C\) such that the four lines through each point of \(G\) are required, each line in these four families contains at most one point of \(G\), and
\[
|G|\ge N-\sum_{j=1}^4(N-R_j).
\]
\end{lemma}

\begin{proof}

For family \(j\), choose one point of \(C\) on each required line and let
\(G_j\) be the resulting set of \(R_j\) representatives. Set
\(G=\bigcap_{j=1}^4G_j\). At most \(\sum_{j=1}^4(N-R_j)\) points are
excluded from this intersection, by the union bound on the complements.
The containment \(G\subseteq G_j\) gives the required-line property and
distinctness in family \(j\).
\end{proof}

For the rest of this section and Sections~\ref{sec:4}--\ref{sec:6}, let
\(C\) be a Type-A \(p\)-cover and write \(N=|C|=2k+r\).
There are \(h\in\{2k,2k+1\}\) required rows and the same number of required
columns, so \(r\ge0\). The required diagonal counts are
\[
R_D=2e+1+2u,\qquad R_S=2f+1+2u.
\]

The Type-A inequality~\eqref{eq:2.1} gives
\(2N-R_D-R_S\le2r+2\). The two orthogonal families contribute at most
\(2(N-h)\le2r\) further exclusions. Lemma~\ref{lem:3.1} therefore gives
\begin{equation}\label{eq:3.1}
|G|\ge2k-3r-2.
\end{equation}

The capacity estimate in Section~\ref{sec:6} uses \(G\). The moment estimate
below is taken over all of \(C\), with repeated line indices counted with
multiplicity.

\section{A second-moment stability estimate}\label{sec:4}

The second moments constrain the lengths of the two required diagonal
families. The output of this step is Lemma~\ref{lem:4.1}: its rational
range bounds specify the domains of the line-weight certificate in
Section~\ref{sec:5}.

Use half-coordinates \(x=X/2,y=Y/2\) and half-diagonal indices
\(d=(Y-X)/2,s=(X+Y)/2\). The sum \(R(k)\) of the squared required half-row
indices is
\begin{equation}\label{eq:4.1}
R(k)=\begin{cases}
\frac23k^3+k^2+\frac k3,&h=2k+1,\\
\frac23k^3-\frac k6,&h=2k.
\end{cases}
\end{equation}

Both formulas hold for integer and half-integer \(k\), and the same sum
applies to columns. The required difference- and sum-diagonal moments are
\(M(e,u)\) and \(M(f,u)\), where
\begin{equation}\label{eq:4.2}
M(t,u)=2\sum_{j=1}^t j^2+2\sum_{j=1}^u(t+2j)^2
=\frac{(t+2u)^3+t^3}{3}+(t+2u)^2+\frac{t+4u}{3}.
\end{equation}

The identity \(d^2+s^2=2x^2+2y^2\) holds at every queen and underlies the
Parallelogram Law (\cite[Theorem~4]{ow2001};
\cite[Theorem~1]{weakley2002lower}). Count each required line once and treat
all remaining occurrences, including repetitions, as extras. Extra half-row
and half-column indices have absolute value at most \(k\), so the total
diagonal moment lies between \(4R(k)\) and \(4R(k)+4(N-h)k^2\).
Removing the extra diagonal occurrences costs at most \(4k^2\) per occurrence.
The bounds \(N-h\le r\) and \(2N-R_D-R_S\le2r+2\), together with
\eqref{eq:4.1}, now give
\begin{equation}\label{eq:4.3}
\frac83k^3-(8r+8)k^2-\frac23k
\le M(e,u)+M(f,u)
\le\frac83k^3+(4r+4)k^2+\frac43k.
\end{equation}

A rotation of the board allows the choice \(e\ge f\). To separate the mean
diagonal scales from their asymmetry, put
\[
\Delta=e-f,\qquad b=\frac{e+f}{4},\qquad a=b+u,\qquad v=a+b.
\]

The line counts imply
\begin{equation}\label{eq:4.4}
a\ge b\ge0,\quad k-1\le v\le k+\frac{r-1}{2},\quad0\le\Delta\le r+1.
\end{equation}

The lower bound on \(v\) comes from the Type-A condition, and its upper
bound comes from \(R_D+R_S\le2N\). For the last inequality, combine
\(2v+\Delta+1=R_D\le2k+r\) with \(v\ge k-1\).
Expanding~\eqref{eq:4.2} then separates the cubic terms from the error terms:
\begin{equation}\label{eq:4.5}
M(e,u)+M(f,u)=\frac{16}{3}(a^3+b^3)+v\Delta^2+8a^2+
\frac{\Delta^2}{2}+\frac{8a-4b}{3}.
\end{equation}

The leading term explains the two scales used in the certificate. Along a
sequence with \(r=o(k)\), equations~\eqref{eq:4.3}--\eqref{eq:4.5} give
\(a/k+b/k\to1\) and \((a/k)^3+(b/k)^3\to1/2\). Since \(a\ge b\),
the normalized pair converges to
\(((3+\sqrt3)/6,(3-\sqrt3)/6)\). The next lemma gives explicit rational
bounds valid throughout the finite range needed below.

\begin{lemma}\label{lem:4.1}
If \(k\ge100000\) and \(r\le k/4000\), then
\begin{equation}\label{eq:4.6}
\frac{e+2u}{2k}<\frac{789}{1000},\qquad
\frac{e}{2k}<\frac{53}{250}.
\end{equation}
\end{lemma}

\begin{proof}

Set
\[
\eta=\frac1{4000},\quad\tau=\frac1{100000},\quad
v_-=1-\tau,\quad v_+=1+\eta/2,\quad h_0=\eta+\tau,
\]
and write \(A=a/k,B=b/k\). The line-count and moment bounds
\eqref{eq:4.3}--\eqref{eq:4.5} give uniform constraints
\[
v_-\le A+B\le v_+,\qquad L\le\frac{16}{3}(A^3+B^3)\le U.
\]
The upper constant follows from the moment upper bound. For the lower
constant, the remaining terms of~\eqref{eq:4.5} are bounded using
\(\Delta/k\le h_0\), \(A\le v_+\), and \(1/k\le\tau\).
Explicitly,
\[
\begin{aligned}
U&=\frac83+4\eta+4\tau+\frac43\tau^2,\\
L&=\frac83-8\eta-8\tau-\frac23\tau^2-
v_+h_0^2-8v_+^2\tau-\frac{h_0^2\tau}{2}-\frac83v_+\tau^2.
\end{aligned}
\]

Choose comparison points \(A_0=157787/200000\) and
\(B_0=42387/200000\). At these points, exact rational evaluation gives the
strict margins
\[
\frac{16}{3}\bigl[A_0^3+(v_--A_0)^3\bigr]-U
=\frac{338621680389}{250000000000000}>0,
\]
\[
L-\frac{16}{3}\bigl[(v_+-B_0)^3+B_0^3\bigr]
=\frac{556961968703}{250000000000000}>0.
\]

Suppose \(A\ge A_0\). Since \(B\ge v_--A\) and
\(t^3+(v_--t)^3\) is increasing for \(t\ge A_0>v_-/2\), the first margin
forces \(16(A^3+B^3)/3>U\). Thus \(A<A_0\).
Similarly, if \(B\ge B_0\), then \(B\le(A+B)/2\le v_+/2\). The function
\((v_+-t)^3+t^3\) is decreasing on \([B_0,v_+/2]\), so the second margin
forces \(16(A^3+B^3)/3<L\). Hence \(B<B_0\).

Finally, \(A_0+h_0/4=789/1000\), \(B_0+h_0/4=53/250\), and \(\Delta/k\le h_0\). Since \((e+2u)/(2k)=A+\Delta/(4k)\) and \(e/(2k)=B+\Delta/(4k)\), \eqref{eq:4.6} follows.
\end{proof}

\section{A rational geometric certificate}\label{sec:5}

Assign nonnegative weights to the four lines through each queen of \(G\).
A pointwise inequality gives a lower bound on the weight at each queen,
while distinctness of the line indices gives an upper bound on the total.
The two bounds form a weak-duality certificate for the line-capacity
constraints. Lemma~\ref{lem:5.1} supplies the geometric inequality and
integral bound needed for the Type-A argument in Section~\ref{sec:6}.
In this section, the arguments \(x,y\) are normalized
coordinates. The application to queen coordinates is given in
Section~\ref{sec:6}.

Let \(q_L=789/1000\) and \(q_S=53/250\).

\begin{lemma}\label{lem:5.1}
There are nonnegative piecewise-linear functions \(F\) on \([0,1]\), \(P_L\) on \([0,2q_L]\), and \(P_S\) on \([0,2q_S]\) with the following properties. For \(j\in\{L,S\}\), whenever \(|x|,|y|\le1\) and \(|x-y|,|x+y|\le2q_j\),
\begin{equation}\label{eq:5.1}
F(|x|)+F(|y|)+P_j(|x-y|)+P_j(|x+y|)\ge T=1000000.
\end{equation}

Their weighted integral is
\begin{equation}\label{eq:5.2}
J=4\int_0^1F+2\int_0^{2q_L}P_L+2\int_0^{2q_S}P_S
=\frac{399330879}{200}<2T.
\end{equation}

Their total variations are \((918849,900105,372693)\), and their respective maxima are \((500000,781301,372693)\).
\end{lemma}

\begin{proof}[Verification of the certificate]

The knot values in Appendix~\ref{app:knots}, together with linear
interpolation, specify all three functions. Their nonnegativity follows
from the nonnegative knot values. The expression in~\eqref{eq:5.1} is
invariant under sign changes and interchange of \(x,y\), so it suffices to
consider \(0\le y\le x\le1\), \(x+y\le2q_j\). Subdivide this polygon by
the knot lines of \(x,y,x-y,x+y\), including its boundary lines
\(y=0\), \(x=y\), \(x=1\), and \(x+y=2q_j\).
The expression is affine on each cell, and its minimum is attained at a
cell vertex.

The line intersections inside the polygon give 1562 vertices for \(j=L\)
and 187 for \(j=S\). Exact rational evaluation gives a minimum of
\(1000000\) over each class. Trapezoidal integration gives~\eqref{eq:5.2};
the variations and maxima follow from successive knot values.
\end{proof}

Figure~\ref{fig:weights} shows the three profiles on a common vertical scale.
The supports of the diagonal weights correspond to the long and short
ranges supplied by Lemma~\ref{lem:4.1}.
\begin{figure}[tbp]
\centering
\begin{minipage}{0.32\textwidth}
\centering
\begin{tikzpicture}
\begin{axis}[
  width=\linewidth,height=4.8cm,
  axis lines=left,axis line style={thin},
  xmin=0,xmax=1,ymin=0,ymax=0.85,
  xtick={0,0.5,1},ytick={0,0.4,0.8},
  scaled ticks=false,tick label style={font=\scriptsize},
  xlabel={$t$},ylabel={$H(t)/T$},
  label style={font=\small},title={$F$},title style={font=\small},
  clip=false]
\addplot[black,semithick,no marks] coordinates {
(0,0.5) (0.025,0.5) (0.05,0.489602) (0.075,0.479203) (0.1,0.458405) (0.125,0.437608) (0.15,0.406411) (0.175,0.375215) (0.2,0.33362) (0.225,0.292025) (0.25,0.240031) (0.275,0.188037) (0.3,0.157078) (0.325,0.128899) (0.35,0.1035) (0.375,0.08088) (0.4,0.061041) (0.425,0.043981) (0.45,0.029702) (0.475,0.018202) (0.5,0.009482) (0.525,0.003541) (0.55,0.000381) (0.575,0) (0.6,0.0024) (0.625,0.007579) (0.65,0.015538) (0.675,0.026277) (0.7,0.039796) (0.725,0.056094) (0.75,0.075173) (0.775,0.103277) (0.8,0.131382) (0.825,0.149087) (0.85,0.179285) (0.875,0.212263) (0.9,0.24802) (0.925,0.286558) (0.95,0.327875) (0.975,0.371972) (1,0.418849)
};
\end{axis}
\end{tikzpicture}
\end{minipage}\hfill
\begin{minipage}{0.32\textwidth}
\centering
\begin{tikzpicture}
\begin{axis}[
  width=\linewidth,height=4.8cm,
  axis lines=left,axis line style={thin},
  xmin=0,xmax=1.578,ymin=0,ymax=0.85,
  xtick={0,0.5,1,1.5},ytick={0,0.4,0.8},
  scaled ticks=false,tick label style={font=\scriptsize},
  xlabel={$t$},
  label style={font=\small},title={$P_L$},title style={font=\small},
  clip=false]
\addplot[black,semithick,no marks] coordinates {
(0,0.781301) (0.025,0.781301) (0.05,0.778521) (0.075,0.775742) (0.1,0.770182) (0.125,0.764622) (0.15,0.756283) (0.175,0.747943) (0.2,0.736824) (0.225,0.725704) (0.25,0.711805) (0.275,0.697906) (0.3,0.681227) (0.325,0.664548) (0.35,0.645089) (0.375,0.62563) (0.4,0.603391) (0.425,0.581152) (0.45,0.556133) (0.475,0.531115) (0.5,0.503316) (0.525,0.475518) (0.55,0.444939) (0.575,0.414361) (0.6,0.381003) (0.625,0.347645) (0.65,0.311507) (0.675,0.275369) (0.7,0.250918) (0.725,0.231666) (0.75,0.212414) (0.775,0.198362) (0.8,0.18431) (0.825,0.175457) (0.85,0.166604) (0.875,0.162951) (0.9,0.159298) (0.925,0.170797) (0.95,0.182297) (0.975,0.191017) (1,0.199737) (1.025,0.205678) (1.05,0.211618) (1.075,0.214778) (1.1,0.217939) (1.125,0.218319) (1.15,0.2187) (1.175,0.2163) (1.2,0.213901) (1.225,0.208722) (1.25,0.203543) (1.275,0.195584) (1.3,0.187625) (1.325,0.176886) (1.35,0.166147) (1.375,0.152628) (1.4,0.139109) (1.425,0.122811) (1.45,0.106512) (1.475,0.087434) (1.5,0.068355) (1.525,0.046497) (1.55,0.024639) (1.575,0) (1.578,0)
};
\end{axis}
\end{tikzpicture}
\end{minipage}\hfill
\begin{minipage}{0.32\textwidth}
\centering
\begin{tikzpicture}
\begin{axis}[
  width=\linewidth,height=4.8cm,
  axis lines=left,axis line style={thin},
  xmin=0,xmax=0.424,ymin=0,ymax=0.85,
  xtick={0,0.2,0.4},ytick={0,0.4,0.8},
  scaled ticks=false,tick label style={font=\scriptsize},
  xlabel={$t$},
  label style={font=\small},title={$P_S$},title style={font=\small},
  clip=false]
\addplot[black,semithick,no marks] coordinates {
(0,0) (0.025,0) (0.05,0.0052) (0.075,0.010399) (0.1,0.020798) (0.125,0.031197) (0.15,0.046795) (0.175,0.062393) (0.2,0.083191) (0.225,0.103988) (0.25,0.129985) (0.275,0.155982) (0.3,0.187179) (0.325,0.218375) (0.35,0.254771) (0.375,0.291167) (0.4,0.332762) (0.424,0.372693)
};
\end{axis}
\end{tikzpicture}
\end{minipage}
\caption{The piecewise-linear weights of Lemma~\ref{lem:5.1}. Each panel shows $H(t)/T$ for the function named above, with $T=10^6$ and a common vertical scale. The horizontal ranges are $[0,1]$, $[0,789/500]$, and $[0,53/125]$. Exact knot values appear in Appendix~\ref{app:knots}.}
\label{fig:weights}
\end{figure}
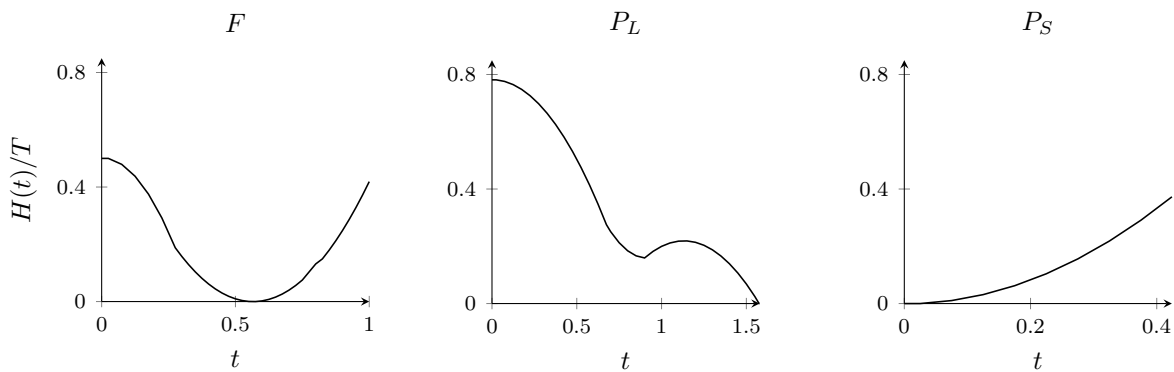

\section{The Type-A density bound}\label{sec:6}

Applying the certificate to the representative set \(G\) completes the
Type-A part of Theorem~\ref{thm:2.1}. The moment bounds place each selected
queen in one of the certificate domains, and a finite lattice estimate bounds the
total available line weight.

The required half-difference indices split into a long parity class
\(d\equiv e\pmod2\), of radius \(e+2u\), and the other parity class,
of radius at most \(e\). The required half-sum indices split in the same
way, with \(f\) in place of \(e\).

For a queen of the extracted set \(G\), both original coordinates have parity \(p\). Thus \(d,s\) are integers and
\[
d+s=Y\equiv p\equiv e+f\pmod2.
\]

Consequently, the two diagonal indices are either both long or both short, for
either choice of \(p\). At the normalized position
\((X/(2k),Y/(2k))\), the diagonal arguments in~\eqref{eq:5.1} are
\(|d|/k\) and \(|s|/k\). Their bounds follow from
Lemma~\ref{lem:4.1} and \(f\le e\). Lemma~\ref{lem:5.1} therefore assigns
weight at least \(T\) to every point of \(G\).

The total available weight is bounded by a finite trapezoidal estimate.
Let \(H\ge0\) be piecewise affine on \([0,R]\), with knots
\(0=t_0<\cdots<t_m=R\), and put
\[
M(H)=\sum_{i=0}^{m-1}\max\{H(t_i),H(t_{i+1})\}.
\]
On one affine segment, the sum over consecutive lattice points of spacing
\(\sigma\) equals the trapezoidal area between the first and last points
divided by \(\sigma\), plus half the sum of their values. Nonnegativity
bounds this by the whole segment area divided by \(\sigma\), plus its
endpoint maximum. The same bound holds for any subset of the lattice points, including the empty
set. Applying it separately to every segment and its reflection gives
\begin{equation}\label{eq:6.1}
\sum H(|t|)\le\frac{2}{\sigma}\int_0^R H(t)\,dt+2M(H).
\end{equation}
Here the sum is over any finite subset of a translated lattice in
\([-R,R]\). Counting a shared endpoint twice preserves the upper bound.

Normalized required row and column indices have spacing \(1/k\), and each
diagonal parity class has spacing \(2/k\). Apply~\eqref{eq:6.1} to the two
orthogonal sums and the four diagonal sums: a long and a short sum in each
diagonal direction. Distinctness of the line indices in \(G\) gives
\[
|G|T\le kJ+E,\qquad
E=4\bigl(M(F)+M(P_L)+M(P_S)\bigr)=134056704.
\]
The area terms add to \(kJ\). The endpoint-maximum sums for
\(F,P_L,P_S\) are \(7988269\), \(23219032\), and \(2306875\),
respectively, as calculated from Appendix~\ref{app:knots}.

Together with~\eqref{eq:3.1}, this yields \((2k-3r-2)T\le kJ+E\).
For \(k\ge100000\) and \(r\le k/4000\), the difference has the following
strictly positive lower bound, with \(\eta=1/4000\) and \(\tau=1/100000\):
\[
\frac{(2k-3r-2)T-kJ-E}{k}
\ge (2-3\eta)T-J-(E+2T)\tau
=\frac{30875949}{25000}>0.
\]

The contradiction proves \(N>(2+1/4000)k\) for Type A.

\section{Type B and the central-queen extension}\label{sec:7}

This section completes Theorem~\ref{thm:2.1}. A moment comparison treats
Type B, and the central-queen extension covers configurations with an
unoccupied long diagonal. The resulting parity-cover theorem is the input
to the completion argument in Section~\ref{sec:8}.

For Type B, the second moments alone give the desired bound. Let \(C\)
have \(N=2k+r\) queens. Its central required half-diagonal families have
sizes \(2e+1\) and \(2f+1\). The inequality \(e+f\ge2k-1\) leaves at
most \(2r\) additional diagonal occurrences, and the individual line counts
give \(e,f\le k+(r-1)/2\).

The total diagonal moment is at least \(4R(k)\). It is bounded above by
\(M(e,0)+M(f,0)\) plus at most \(8rk^2\) from the extra occurrences.
For \(k\ge100000\) and \(r\le k/4000\), put \(\eta=1/4000\),
\(\tau=1/100000\), and \(v_+=1+\eta/2\). Dividing the moment comparison
by \(k^3\) gives
\[
\frac83-\frac23\tau^2
\le\frac43v_+^3+2v_+^2\tau+\frac23v_+\tau^2+8\eta.
\]

The left side exceeds the right side by
\(53232530627883/40000000000000\), excluding this range of \(r\) for
Type B as well.

Every \(p\)-cover meeting both long diagonals is classified by \eqref{eq:2.1}. The preceding arguments establish
\[
|C|>\left(2+\frac1{4000}\right)k=\frac{8001}{16000}(n-1).
\]

To cover the general case, take \(C'=C\cup\{(0,0)\}\). The added
central queen occupies both long diagonals and preserves all coverage
conditions. Since \(|C'|\le|C|+1\), applying the stronger estimate to
\(C'\) proves Theorem~\ref{thm:2.1}.

\section{Diagonal completion of relaxed covers}\label{sec:8}

The next reduction turns a relaxed cover into a parity cover on the same
odd board. Lemma~\ref{lem:8.1} bounds the number of added queens, yielding
Theorem~\ref{thm:2.2}. A small-board example illustrates the coverage that
the construction supplies.

A relaxed \(p\)-cover may use additional occupied rows and columns to cover squares with both coordinates of parity \(1-p\). For example, in centered coordinates on \(Q_5\), the set
\[
\{(-1,-1),(1,1),(2,2)\}
\]
is a relaxed 1-cover. The extra row and column through \((2,2)\) cover
\((-2,2)\) and \((2,-2)\), respectively. Adding a queen at the center
gives these two squares diagonal coverage and produces a 1-cover
(Figure~\ref{fig:completion}). The following lemma bounds the cost of this
completion on an arbitrary odd board.
\begin{figure}[htbp]
\centering
\begin{tikzpicture}[x=0.75cm,y=0.75cm]
\foreach \panel in {0,1}{
  \begin{scope}[xshift=\panel*6.8cm]
    \foreach \b in {-2.5,-1.5,-0.5,0.5,1.5,2.5}{
      \draw[very thin] (\b,-2.5)--(\b,2.5);
      \draw[very thin] (-2.5,\b)--(2.5,\b);
    }
    \draw[thin] (-2.5,-2.5) rectangle (2.5,2.5);
    \foreach \i in {-2,-1,0,1,2}{
      \node[font=\scriptsize,below] at (\i,-2.5) {$\i$};
      \node[font=\scriptsize,left] at (-2.5,\i) {$\i$};
    }
    \ifnum\panel=0
      \draw[densely dashed,semithick] (-2,2)--(2,2)--(2,-2);
      \node[font=\small] at (0,-3.35) {(a) Relaxed $1$-cover $C$};
    \else
      \draw[densely dashed,semithick] (-2.5,2.5)--(2.5,-2.5);
      \draw[fill=white,semithick] (0,0) circle[radius=3.1pt];
      \node[font=\small] at (0,-3.35) {(b) $1$-cover $C\cup\{(0,0)\}$};
    \fi
    \foreach \x/\y in {-1/-1,1/1,2/2}{\fill (\x,\y) circle[radius=2.8pt];}
    \foreach \x/\y in {-2/2,2/-2}{
      \draw[semithick] (\x-0.13,\y-0.13)--(\x+0.13,\y+0.13);
      \draw[semithick] (\x-0.13,\y+0.13)--(\x+0.13,\y-0.13);
    }
  \end{scope}
}
\end{tikzpicture}
\caption{Diagonal completion on $Q_5$. Filled dots represent $C=\{(-1,-1),(1,1),(2,2)\}$; the open circle marks the added queen. Crosses indicate $(-2,2)$ and $(2,-2)$. In (a), their coverage is along the dashed row and column. In (b), both also lie on the newly occupied sum diagonal $X+Y=0$.}
\label{fig:completion}
\end{figure}
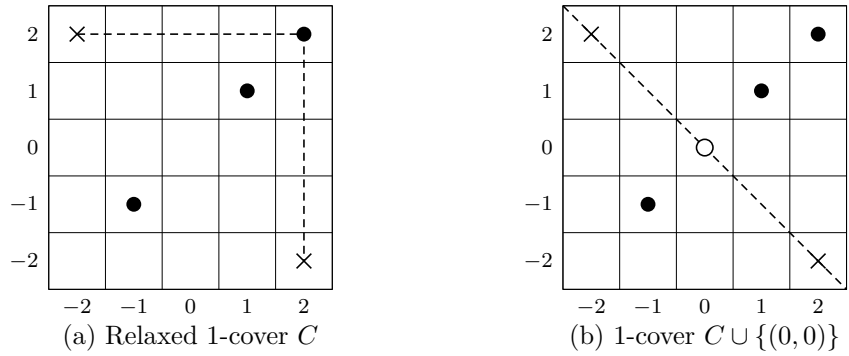

\begin{lemma}[diagonal completion]\label{lem:8.1}
Let \(C\) be a relaxed \(p\)-cover on an odd board. Let \(A\) and \(B\) be the numbers of occupied columns and rows, respectively, of parity \(1-p\). There is a \(p\)-cover \(P\supseteq C\) on the same board with
\begin{equation}\label{eq:8.1}
|P\setminus C|\le2(A+B).
\end{equation}
\end{lemma}

\begin{proof}

Write \(n=2L+1\), use coordinates \(-L,\ldots,L\), and put \(q=1-p\). A square \((X,Y)\) with both coordinates of parity \(q\) has integer half-diagonal indices
\[
d=(Y-X)/2,\qquad s=(X+Y)/2,\qquad (X,Y)=(s-d,s+d).
\]

Let \(\mathcal M_D\) and \(\mathcal M_S\) contain the half-indices of
unoccupied board diagonals with even difference and sum indices,
respectively. A square of the indicated parity requires diagonal completion
precisely when
\begin{equation}\label{eq:8.2}
d\in\mathcal M_D,\quad s\in\mathcal M_S,\quad
d+s\equiv q\pmod2,\quad |d|+|s|\le L.
\end{equation}

Here the board condition follows from
\(\max\{|s-d|,|s+d|\}=|d|+|s|\).

Condition~\eqref{eq:8.2} defines a bipartite graph on the missing difference
and sum diagonals. It splits into two parity subgraphs: for \(b=0,1\), take
\(d\equiv b\pmod2\) and \(s\equiv q-b\pmod2\). Within either subgraph,
ordering indices by absolute value gives nested neighborhoods.

Every edge represents a square whose domination by \(C\) is supplied by one
of the \(A+B\) additional occupied orthogonals. A fixed sum diagonal meets
each such orthogonal in at most one square, so every sum-diagonal vertex
has degree at most \(A+B\). It remains to bound the number of difference
vertices incident with an edge.

In a parity subgraph containing an edge, choose a missing sum index \(s_0\) of
minimum absolute value. Every nonisolated difference vertex \(d\) is
adjacent to \(s_0\), since any existing neighbor \(s\) gives
\[
|d|+|s_0|\le |d|+|s|\le L.
\]

Consequently, the number of nonisolated difference vertices in this subgraph is at most \(A+B\). Across both subgraphs there are at most \(2(A+B)\) such vertices.

For each such index \(d\), add a queen at \((-d,d)\). Incidence with an
edge gives \(|d|\le L\), so this square lies on the board. Its difference
diagonal was unoccupied, and different indices give different new queens.
Every square represented by an edge now has diagonal coverage. The required
orthogonals remain occupied, so the enlarged set is a \(p\)-cover and
satisfies~\eqref{eq:8.1}.
\end{proof}

\begin{proof}[Proof of Theorem~\ref{thm:2.2}]

Put \(N=|C|\), and let \(h\) be the common number of required rows and
columns. There are \(h+A\) occupied columns and \(h+B\) occupied rows.
Thus \(A,B\le N-h\), and Lemma~\ref{lem:8.1} supplies a \(p\)-cover of
size at most \(N+4(N-h)=5N-4h\). Applying Theorem~\ref{thm:2.1} gives
\begin{equation}\label{eq:8.3}
5N-4h>c_0(n-1)-1,
\qquad
N>\frac{c_0(n-1)+4h-1}{5}.
\end{equation}

Using \(h\ge(n-1)/2\) proves \eqref{eq:2.3}.
\end{proof}

The same calculation transfers any strict bound \(|P|>f(n)\) for
\(p\)-covers to the bound \(|C|>(f(n)+4h)/5\) for relaxed \(p\)-covers.
In particular, the asymptotic excess above one half is multiplied by one
fifth under this completion estimate.

\section{Parity alignment and orthodoxy defect}\label{sec:9}

The remaining parity choices are handled by a small board extension. This
preserves the leading density coefficient and changes only the additive
constant, completing the proof of the principal result,
Theorem~\ref{thm:2.3}. The final argument identifies the exact cost of an
orthodox extension and derives Theorem~\ref{thm:2.4}.

\begin{lemma}[board extension]\label{lem:9.1}
An orthodox cover of \(Q_n\) extends, after a translation, to a relaxed \(p\)-cover of an odd board \(Q_m\), where \(0\le m-n\le2\). Exactly \(m-n\) additional queens suffice.
\end{lemma}

\begin{proof}

Use board coordinates \(0,\ldots,n-1\), and let the occupied parity classes be \(p\) for columns and \(q\) for rows. Choose
\[
m=\begin{cases}
n,& n\text{ odd and }p=q,\\
n+1,& n\text{ even},\\
n+2,& n\text{ odd and }p\ne q,
\end{cases}
\qquad a=(p-q)\bmod2\in\{0,1\}.
\]

Translate every queen by \((a,0)\), embedding the original board as
\([a,a+n-1]\times[0,n-1]\) inside the \(m\times m\) board. The required column parity becomes \(p+a\equiv q\pmod2\), matching the required row parity.

There are \(d_0=m-n\) new rows and the same number of new columns.
Pair them and place a queen at each of the \(d_0\) intersections.
Every square outside the embedded board lies in one of these occupied lines;
the translated original queens still dominate the embedded board.
All rows and columns of parity \(q\) are occupied, so the result is a
relaxed \(q\)-cover on the odd board.
\end{proof}

\begin{proof}[Proof of Theorem~\ref{thm:2.3}]

Apply Lemma~\ref{lem:9.1} and Theorem~\ref{thm:2.2} to obtain
\[
|C|+d_0>c_1(m-1)-\frac15.
\]

Since \(m=n+d_0\), \(d_0\le2\), and \(c_1<1\),
\[
|C|>c_1n-c_1-\frac15-(1-c_1)d_0
\ge c_1n-\left(\frac{11}{5}-c_1\right)>c_1n-2.
\]

The final comparison uses \(11/5-c_1=135999/80000<2\). This also proves \eqref{eq:1.1}.
\end{proof}

\begin{proof}[Proof of Theorem~\ref{thm:2.4}]

Fix parity choices \(p,q\) attaining the minimum in \(\delta(C)\).
Pair the unoccupied columns of parity \(p\) with the unoccupied rows of
parity \(q\), adding a queen at each intersection. Place one further queen
on each remaining unpaired line. Exactly \(\delta(C)\) queens are added,
producing an orthodox superset.

For any prescribed parity pair, at least
\(\max\{a_p(C),b_q(C)\}\) additions are needed: a queen occupies one row
and one column. Minimizing this necessary number over the pairs proves the
interpretation of \(\delta(C)\) as the exact extension cost.
Theorem~\ref{thm:2.3}, applied to the constructed superset, gives~\eqref{eq:2.5}.
\end{proof}

For any sequence of dominating sets \(C_n\) with \(|C_n|=n/2+o(n)\), Theorem~\ref{thm:2.4} yields
\begin{equation}\label{eq:9.1}
\liminf_{n\to\infty}\frac{\delta(C_n)}n\ge\frac1{80000}.
\end{equation}

Here the numerator counts the minimum number of additional queens needed
to extend the given dominating set to an orthodox cover.
Thus every such sequence has a linear asymptotic cost of extension to the
orthodox class.

\section{Finite seeds and amplification limits}\label{sec:10}

With the density proofs complete, the finite seed and the growing-seed
consequence can be treated separately. The \(Q_{221}\) configuration
from~\cite{kong2026} illustrates classical amplification at a fixed seed
order. Its four line-index multisets verify the hypotheses of that
construction. The density theorem then constrains the coefficients along
growing sequences of admissible seeds.

\subsection{An independent border-free cover of \texorpdfstring{\(Q_{221}\)}{Q221}}

Use centered integer coordinates \(-110\le X,Y\le110\), and let
\(C_{221}\) be the 111 points reproduced from~\cite{kong2026} in
Appendix~\ref{app:coordinates}. The four line indices of a queen \((X,Y)\)
are \(X,Y,Y-X,X+Y\).

\begin{proposition}\label{prop:10.1}
The set \(C_{221}\) is an independent, border-free Type-A 1-cover with parameters
\[
(e,f,u)=(24,23,31).
\]

In particular, \(\gamma(Q_{221})=i(Q_{221})=111\).
\end{proposition}

Figure~\ref{fig:q221} shows the placement.
\begin{figure}[tbp]
\centering
\begin{tikzpicture}
\begin{axis}[
  width=9cm,height=9cm,scale only axis,axis equal image,
  xmin=-110.5,xmax=110.5,ymin=-110.5,ymax=110.5,
  axis lines=box,axis line style={thin},
  xtick={-110,-55,0,55,110},ytick={-110,-55,0,55,110},
  tick align=outside,tick label style={font=\footnotesize},
  xlabel={$X$},ylabel={$Y$},label style={font=\small},
  clip=true]
\addplot[black,thin,densely dashed,no marks] coordinates {(-110.5,-110.5) (110.5,110.5)};
\addplot[black,thin,densely dashed,no marks] coordinates {(-110.5,110.5) (110.5,-110.5)};
\addplot[only marks,mark=*,mark size=1.15pt,black] coordinates {
(-109,-61) (-107,-43) (-105,-37) (-103,57) (-101,63) (-99,33) (-97,-57) (-95,-35) (-93,-9) (-91,17) (-89,11) (-87,37) (-85,43) (-83,-3) (-81,75) (-79,69) (-77,-85) (-75,-47) (-73,-41) (-71,101) (-69,-49) (-67,-79) (-65,79) (-63,105) (-61,35) (-59,-107) (-57,-101) (-55,81) (-53,-45) (-51,-75) (-49,55) (-47,93) (-45,67) (-43,109) (-41,-39) (-39,-67) (-37,-73) (-35,-99) (-33,-105) (-31,-13) (-29,91) (-27,7) (-25,13) (-23,23) (-21,5) (-19,97) (-17,-11) (-15,27) (-13,-23) (-11,-29) (-9,83) (-7,-83) (-5,25) (-3,-91) (-1,87) (1,-33) (3,-27) (4,4) (5,-87) (7,29) (9,-17) (11,21) (13,85) (15,9) (17,-21) (19,95) (21,19) (23,-19) (25,3) (27,-97) (29,15) (31,-15) (33,-63) (35,59) (37,51) (39,-93) (41,-103) (43,99) (45,89) (47,-81) (49,53) (51,-109) (53,65) (55,107) (57,-55) (59,-77) (61,77) (63,-53) (65,61) (67,103) (69,-71) (71,39) (73,-95) (75,-89) (77,-7) (79,-25) (81,1) (83,-65) (85,45) (87,71) (89,-31) (91,31) (93,73) (95,-5) (97,-59) (99,47) (101,-51) (103,-69) (105,49) (107,-1) (109,41)
};
\draw[thin] (axis cs:4,4)--(axis cs:23,14) node[anchor=west,font=\footnotesize] {$(4,4)$};
\end{axis}
\end{tikzpicture}
\caption{The $111$ queens of $C_{221}$ in Proposition~\ref{prop:10.1}. Dots give the exact integer positions, the frame marks the board boundary, and dashed lines show the two long diagonals. The labelled point $(4,4)$ is the unique queen in an even row or column.}
\label{fig:q221}
\end{figure}
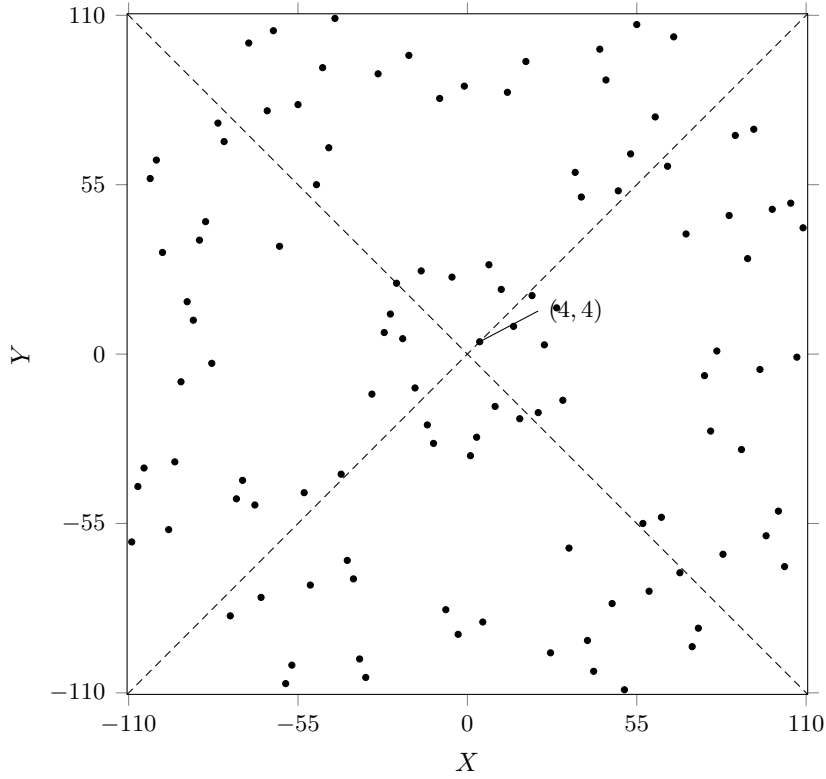

\begin{proof}

Direct evaluation of the four line-index multisets of Appendix~\ref{app:coordinates} gives the same multiset for columns and rows:
\[
\mathcal O=\{-109,-107,\ldots,109\}\uplus\{4\}.
\]

The difference and sum indices are, respectively,
\begin{equation}\label{eq:10.1}
\begin{aligned}
\mathcal D&=\{2j:-24\le j\le24\}
 \uplus\{\pm(48+4j):1\le j\le31\},\\
\mathcal S&=\{2j:-23\le j\le23\}
 \uplus\{\pm(46+4j):1\le j\le31\}
 \uplus\{-80,88\}.
\end{aligned}
\end{equation}

Each of the four multisets has 111 distinct entries, so the queens are
pairwise nonattacking. The row and column indices also verify that every
queen lies in the interior and that every odd row and column is occupied.

The central intervals in~\eqref{eq:10.1} have maximal radii 24 and 23,
because the next difference and sum indices, \(\pm50\) and \(\pm48\),
are absent. The outer families both continue through \(j=31\); the next
pairs, \(\pm176\) and \(\pm174\), are absent. This gives the stated
values of \(e,f,u\). Both long diagonals are occupied, and
\[
e+f=47\equiv1\pmod2,\qquad e+f+2u=109\ge108=\frac{221-5}{2}.
\]

The Type-A characterization~\eqref{eq:2.1} now gives the 1-cover property.
The lower bound \(\gamma(Q_{4k+1})\ge2k+1\), recorded
in~\cite[equation~(2)]{ow2001}, gives \(\gamma(Q_{221})\ge111\). Hence
\[
111\le\gamma(Q_{221})\le i(Q_{221})\le|C_{221}|=111.\qedhere
\]
\end{proof}

The verified seed can now be substituted into the classical amplification
formula.

\subsection{The classical amplification formula}

The applicable border-free branch of the amplification theorem has the
following hypotheses. Let \(D\) be a Type-A \(p\)-cover of an odd board of
order \(n\), with \(d\) queens on interior squares and with both long
diagonals occupied. If either \(p=1\) and \(n\equiv1\pmod4\), or \(p=0\)
and \(n\equiv3\pmod4\), then (\cite[Theorem~5]{ow2001};
\cite{weakley2002upper})
\begin{equation}\label{eq:10.2}
\gamma(Q_N)\le\frac{d+1}{n}N+O_D(1).
\end{equation}

If \(D\) is independent, the same theorem gives
\begin{equation}\label{eq:10.3}
i(Q_N)\le\frac{d+2}{n}N+O_D(1).
\end{equation}

In both statements, \(D\) is fixed while \(N\) tends to infinity, and the
bounded term may depend on \(D\). Proposition~\ref{prop:10.1} supplies all
the hypotheses for \(n=221\), \(d=111\), yielding~\eqref{eq:1.3}.

These upper bounds concern ordinary and independent domination. Applying
the orthodox lower bound to an amplified configuration requires the
corresponding parity-occupancy conditions as well as the board-order
threshold.

The coefficients \(30/59\) and \(91/177\) were obtained by Neuhaus \cite[Corollary~3.24]{neuhaus2009} from an independent, border-free Type-A 1-cover of \(Q_{177}\) with 89 queens. Applying the same amplification branch to \(Q_{221}\) gives the respective decreases
\[
\frac{30}{59}-\frac{112}{221}=\frac{22}{13039},\qquad
\frac{91}{177}-\frac{113}{221}=\frac{110}{39117}.
\]

\subsection{A restriction on growing seeds}

Consider next a sequence of admissible seeds whose orders tend to infinity.
This limit varies the seed itself; the preceding subsection fixes the seed
and lets the amplified board grow. The following consequence of
Theorem~\ref{thm:2.1} bounds the coefficients in
\eqref{eq:10.2}--\eqref{eq:10.3} along such sequences.

\begin{corollary}\label{cor:10.2}
Let \(D\) be a seed satisfying the hypotheses of \eqref{eq:10.2}, with \(n\ge400001\). Then
\begin{equation}\label{eq:10.4}
\frac{d+1}{n}>c_0+\frac{1-c_0}{n}>c_0,
\qquad c_0=\frac{8001}{16000}.
\end{equation}

For an independent seed, the coefficient in \eqref{eq:10.3} satisfies
\begin{equation}\label{eq:10.5}
\frac{d+2}{n}>c_0+\frac{2-c_0}{n}>c_0.
\end{equation}

Consequently, along every sequence of such seeds whose orders tend to infinity, the lower limit of either applicable coefficient is at least \(1/2+1/16000\).
\end{corollary}

\begin{proof}

Each seed meets both long diagonals, so the stronger part of
Theorem~\ref{thm:2.1} gives \(d>c_0(n-1)\). Adding one or two and dividing
by \(n\) gives the two finite-order inequalities. Taking lower limits along
the sequence proves the final assertion.
\end{proof}

\section{Discussion}\label{sec:11}

The density theorems have consequences for the zero-cover construction
route, orthodox extension costs, and two conjectures of Neuhaus. These
applications and the remaining questions are considered below.

The density gap for orthodox domination follows from a gap for \(p\)-covers
and a uniform bound on the cost of completion. Its dependence on the
\(p\)-cover estimate and the completion cost is explicit in the transfer
formula following~\eqref{eq:8.3}. Improving either estimate would sharpen
the orthodox bound.

The zero-cover route to asymptotic half-density discussed by
Weakley~\cite[p.~52]{weakley2018} requires arbitrarily large seeds of size
\((n+1)/2\). By Theorem~\ref{thm:2.3}, such seeds are confined to finitely
many orders.

For arbitrary dominating sets, the exact extension cost \(\delta\) measures
the part of the bound supplied by the parity constraints.
Equation~\eqref{eq:9.1} forces this cost to be linear along every sequence
of size \(n/2+o(n)\). Establishing a positive asymptotic gap for unrestricted
queen domination remains an open problem.

Two related conjectures were formulated by Neuhaus \cite[p.~112]{neuhaus2009}. Conjecture 5.2 asserts that every minimum dominating set is orthodox for all sufficiently large \(n\). Conjecture 5.3 asserts, for \(n\notin\{3,11\}\), that
\[
\gamma(Q_n)\in\{\lceil n/2\rceil,\lceil n/2\rceil+1\}.
\]

\begin{corollary}\label{cor:11.1}
Conjectures 5.2 and 5.3 of Neuhaus are incompatible.
\end{corollary}

\begin{proof}

Together, the conjectures would give an orthodox minimum dominating set of
size at most \(n/2+3/2\) for all sufficiently large \(n\).
Theorem~\ref{thm:2.3} gives the strict lower bound \(n/2+n/80000-2\),
which eventually exceeds that upper bound.
\end{proof}

If the near-half upper bound in Conjecture 5.3 holds for arbitrarily large orders, every minimum dominating set at those orders is eventually nonorthodox.

The optimal density coefficient and the smallest valid board order remain
to be determined. The rational margins in Sections~\ref{sec:4} and
\ref{sec:6} give explicit points at which the numerical estimates can be
refined.

\section{Computational verification}\label{sec:12}

The numerical certificates are reproduced by exact finite checks; the
orthodox density theorem is verified formally from its definitions. Their
scope and accompanying data are described below.

The source package contains the line-weight and \(Q_{221}\) data together
with Python checkers using exact arithmetic and the standard library.
Appendices~\ref{app:knots} and~\ref{app:coordinates} list the complete
coefficients and coordinates. The earlier \(Q_{221}\) certificate is also
available in~\cite{kong2026}.

For the line weights, Lemma~\ref{lem:5.1} gives the finite vertex evaluation,
integration, and variation calculations. All these calculations are
reproduced by the accompanying exact checker.
The verifier in the accompanying source package and the repository cited in~\cite{kong2026} checks all \(48{,}841\) board squares and all queen pairs directly from the coordinates.

Theorem~\ref{thm:2.3} is fully formalized in Lean~4.33.0 with
mathlib~v4.33.0. The classification, moment estimates, affine weight
inequalities, lattice capacities, completion, and parity alignment are derived
from the finite chessboard and orthodox-cover definitions. The proof source
is included in the accompanying package. Its final declaration is
\texttt{orthodox\_density\_gap}, and the lattice argument follows
Section~\ref{sec:6}.

\appendix
\section{Complete knot data}\label{app:knots}
The following knot values specify the three functions of
Lemma~\ref{lem:5.1} by linear interpolation.
For \(F\), set \(F_i=F(i/40)\) for \(0\le i\le40\).
For \(P_L\), set \(L_i=P_L(i/40)\) for \(0\le i\le63\), and
\(L_{64}=P_L(789/500)\).
For \(P_S\), set \(S_i=P_S(i/40)\) for \(0\le i\le16\), and
\(S_{17}=P_S(53/125)\).
All 124 values in the following table are integers. Within each function,
indices increase down one pair of columns and then continue in the next pair.

From these knots, the exact checker reconstructs every linearity region and
evaluates its vertices using rational arithmetic. The same data determine the
trapezoidal integrals, endpoint-maximum sums, and comparison margins.
The weights were found by numerical optimization; the displayed rational data
provide a complete finite certificate for their verification.

\begingroup
\small
\setlength{\tabcolsep}{5pt}
\renewcommand{\arraystretch}{1.04}
\begin{center}
\begin{tabular}{@{}rrrrrrrrrrrr@{}}
\toprule
\multicolumn{4}{c}{\(F\)} & \multicolumn{6}{c}{\(P_L\)} & \multicolumn{2}{c}{\(P_S\)}\\
\cmidrule(lr){1-4}\cmidrule(lr){5-10}\cmidrule(lr){11-12}
\(i\)&\(F_i\)&\(i\)&\(F_i\)&\(i\)&\(L_i\)&\(i\)&\(L_i\)&\(i\)&\(L_i\)&\(i\)&\(S_i\)\\
\midrule
0 & 500000 & 21 & 3541 & 0 & 781301 & 22 & 444939 & 44 & 217939 & 0 & 0 \\
1 & 500000 & 22 & 381 & 1 & 781301 & 23 & 414361 & 45 & 218319 & 1 & 0 \\
2 & 489602 & 23 & 0 & 2 & 778521 & 24 & 381003 & 46 & 218700 & 2 & 5200 \\
3 & 479203 & 24 & 2400 & 3 & 775742 & 25 & 347645 & 47 & 216300 & 3 & 10399 \\
4 & 458405 & 25 & 7579 & 4 & 770182 & 26 & 311507 & 48 & 213901 & 4 & 20798 \\
5 & 437608 & 26 & 15538 & 5 & 764622 & 27 & 275369 & 49 & 208722 & 5 & 31197 \\
6 & 406411 & 27 & 26277 & 6 & 756283 & 28 & 250918 & 50 & 203543 & 6 & 46795 \\
7 & 375215 & 28 & 39796 & 7 & 747943 & 29 & 231666 & 51 & 195584 & 7 & 62393 \\
8 & 333620 & 29 & 56094 & 8 & 736824 & 30 & 212414 & 52 & 187625 & 8 & 83191 \\
9 & 292025 & 30 & 75173 & 9 & 725704 & 31 & 198362 & 53 & 176886 & 9 & 103988 \\
10 & 240031 & 31 & 103277 & 10 & 711805 & 32 & 184310 & 54 & 166147 & 10 & 129985 \\
11 & 188037 & 32 & 131382 & 11 & 697906 & 33 & 175457 & 55 & 152628 & 11 & 155982 \\
12 & 157078 & 33 & 149087 & 12 & 681227 & 34 & 166604 & 56 & 139109 & 12 & 187179 \\
13 & 128899 & 34 & 179285 & 13 & 664548 & 35 & 162951 & 57 & 122811 & 13 & 218375 \\
14 & 103500 & 35 & 212263 & 14 & 645089 & 36 & 159298 & 58 & 106512 & 14 & 254771 \\
15 & 80880 & 36 & 248020 & 15 & 625630 & 37 & 170797 & 59 & 87434 & 15 & 291167 \\
16 & 61041 & 37 & 286558 & 16 & 603391 & 38 & 182297 & 60 & 68355 & 16 & 332762 \\
17 & 43981 & 38 & 327875 & 17 & 581152 & 39 & 191017 & 61 & 46497 & 17 & 372693 \\
18 & 29702 & 39 & 371972 & 18 & 556133 & 40 & 199737 & 62 & 24639 &  &  \\
19 & 18202 & 40 & 418849 & 19 & 531115 & 41 & 205678 & 63 & 0 &  &  \\
20 & 9482 &  &  & 20 & 503316 & 42 & 211618 & 64 & 0 &  &  \\
 &  &  &  & 21 & 475518 & 43 & 214778 &  &  &  &  \\
\bottomrule
\end{tabular}
\end{center}
\endgroup

\Needspace{12\baselineskip}
\section{The \texorpdfstring{\(Q_{221}\)}{Q221} coordinate certificate}\label{app:coordinates}
These 111 points, reproduced from~\cite{kong2026}, use the centered integer
coordinates \((X,Y)\in[-110,110]^2\) of Section~\ref{sec:10}.
Read from left to right and then from top to bottom. Their line-index
signatures and amplification properties are verified in
Proposition~\ref{prop:10.1}.

\begingroup
\small
\setlength{\tabcolsep}{8pt}
\renewcommand{\arraystretch}{1.00}
\begin{longtable}{@{}cccccc@{}}
\toprule
\((X,Y)\)&\((X,Y)\)&\((X,Y)\)&\((X,Y)\)&\((X,Y)\)&\((X,Y)\)\\
\midrule
\endhead
\bottomrule
\endfoot
\((-109,-61)\) & \((-107,-43)\) & \((-105,-37)\) & \((-103,57)\) & \((-101,63)\) & \((-99,33)\) \\
\((-97,-57)\) & \((-95,-35)\) & \((-93,-9)\) & \((-91,17)\) & \((-89,11)\) & \((-87,37)\) \\
\((-85,43)\) & \((-83,-3)\) & \((-81,75)\) & \((-79,69)\) & \((-77,-85)\) & \((-75,-47)\) \\
\((-73,-41)\) & \((-71,101)\) & \((-69,-49)\) & \((-67,-79)\) & \((-65,79)\) & \((-63,105)\) \\
\((-61,35)\) & \((-59,-107)\) & \((-57,-101)\) & \((-55,81)\) & \((-53,-45)\) & \((-51,-75)\) \\
\((-49,55)\) & \((-47,93)\) & \((-45,67)\) & \((-43,109)\) & \((-41,-39)\) & \((-39,-67)\) \\
\((-37,-73)\) & \((-35,-99)\) & \((-33,-105)\) & \((-31,-13)\) & \((-29,91)\) & \((-27,7)\) \\
\((-25,13)\) & \((-23,23)\) & \((-21,5)\) & \((-19,97)\) & \((-17,-11)\) & \((-15,27)\) \\
\((-13,-23)\) & \((-11,-29)\) & \((-9,83)\) & \((-7,-83)\) & \((-5,25)\) & \((-3,-91)\) \\
\((-1,87)\) & \((1,-33)\) & \((3,-27)\) & \((4,4)\) & \((5,-87)\) & \((7,29)\) \\
\((9,-17)\) & \((11,21)\) & \((13,85)\) & \((15,9)\) & \((17,-21)\) & \((19,95)\) \\
\((21,19)\) & \((23,-19)\) & \((25,3)\) & \((27,-97)\) & \((29,15)\) & \((31,-15)\) \\
\((33,-63)\) & \((35,59)\) & \((37,51)\) & \((39,-93)\) & \((41,-103)\) & \((43,99)\) \\
\((45,89)\) & \((47,-81)\) & \((49,53)\) & \((51,-109)\) & \((53,65)\) & \((55,107)\) \\
\((57,-55)\) & \((59,-77)\) & \((61,77)\) & \((63,-53)\) & \((65,61)\) & \((67,103)\) \\
\((69,-71)\) & \((71,39)\) & \((73,-95)\) & \((75,-89)\) & \((77,-7)\) & \((79,-25)\) \\
\((81,1)\) & \((83,-65)\) & \((85,45)\) & \((87,71)\) & \((89,-31)\) & \((91,31)\) \\
\((93,73)\) & \((95,-5)\) & \((97,-59)\) & \((99,47)\) & \((101,-51)\) & \((103,-69)\) \\
\((105,49)\) & \((107,-1)\) & \((109,41)\) &  &  &  \\
\end{longtable}
\endgroup

\paragraph{AI-assistance statement.}
ChatGPT 5.6 participated in manuscript preparation and code development.

\end{document}